\documentclass[11pt,leqno,a4paper]{amsart}

\usepackage[]{hyperref}
\hypersetup{
    colorlinks=true,       
    linkcolor=red,          
    citecolor=blue,        
    filecolor=magenta,      
    urlcolor=cyan           
}

\usepackage{amsmath}
\usepackage{amsfonts,amssymb}
\usepackage{enumerate}
\usepackage{mathrsfs}
\usepackage{amsthm}
 
\usepackage{tikz}
\usepackage[cp1251]{inputenc}

\theoremstyle{plain}
\newtheorem{theo}{Theorem}[section]
\newtheorem{prop}[theo]{Proposition}
\newtheorem{lemm}[theo]{Lemma}
\newtheorem{coro}[theo]{Corollary}

\newtheorem{defi}[theo]{Definition}
\theoremstyle{definition}

\newtheorem{remas}[theo]{Remark}

\DeclareSymbolFont{pSoitters}{OT1}{cmr}{m}{sl}
\DeclareMathSymbol{s}{\mathalpha}{pSoitters}{`s}

\def\eps{\varepsilon}

\def\mez{\frac{1}{2}}

\def\xC{\mathbf{C}}
\def\xN{\mathbf{N}}
\def\xR{\mathbf{R}}

\numberwithin{equation}{section}

\date{}
\title{ Domains of dependence for subelliptic wave equations and unique continuation for fractional powers of H\"ormander's operators}

  \author{Nicolas Burq}
  \address{ Universit\'e Paris-Saclay, Laboratoire de Math\'ematique d'Orsay, UMR 8628 du CNRS 91405 Orsay Cedex, France\\ and Institut universitaire de France}
  \email{nicolas.burq@universite-paris-saclay.fr}
  \author{Claude Zuily}
 \address{ Universit\'e Paris-Saclay, Laboratoire de Math\'ematique d'Orsay, UMR 8628 du CNRS 91405 Orsay Cedex, France}
 \email{claude.zuily@universite-paris-saclay.fr}
\begin{document}
  \begin{abstract} We prove the sharp domain of dependence property for solutions to subelliptic wave equations for sums of squares of vector fields satisfying H\"ormander bracket condition. We deduce a unique continuation property for the square root of subelliptic Laplace operators under an additional analyticity condition. Then, with a different, more involved  method, we prove the same result of unique continuation for more general $s$-powers ($0<s<1$).
  \end{abstract}
\maketitle
\section{Introduction and results}
In this note, we are interested in unique continuation properties of fractional powers of H\"ormander's subelliptic sum of squares operators.  In the first part of the paper,  we determine the sharp domain of dependence for solutions of  subelliptic wave equations in $\xR\times \xR^d$ with smooth coefficients, a result of independent interest. As an application we give, in the second part, a rather simple proof of  unique continuation for square roots of  of H\"ormander's operators  with analytic coefficients. Finally, in the third part, we consider the general case of $s$-powers of these  operators  ($0<s<1$) and we prove a similar result of unique continuation.

 Let $X= (X_1, \ldots, X_r)$ be a set of real vector fields  with $C^\infty$ coefficients on   $\xR^d$. 

Following H\"ormander we shall denote by $\mathcal{L}(X_1, \ldots, X_r)$ the Lie algebra generated by $X_1, \ldots, X_r$ and we shall  assume H\"ormander 's condition, that is,
\begin{equation}\label{lie}
  \mathcal{L}(X_1, \ldots, X_r) \text{ has maximal rank at every point of } \xR^d.
  \end{equation}
We shall consider the operator, $P = \sum_{j=1}^r X_j^*X_j$, 
where $X_j^*$ is the adjoint of $X_j$ and denote by $H^1_X$ the closure of $C^\infty_0 ( \mathbb{R}^d)$ for the norm 
$$ \| u \|^2_{H^1_X} = \sum _j \| X_j u \|_{L^2( \mathbb{R}^d)}^2$$
The quadratic form 
$$ Q(u,v) = \sum_j \bigl( X_j u , X_j v \bigr) _{L^2} + \bigl( u, v \bigr) _{L^2}$$ is continuous on $H^1_X$ and Riesz representation theorem shows that the operator $P + \text{Id}$ is an isometry from $H^1_X$ into its dual space $H^{-1}_X$ (which is a distribution space because $C^\infty_0$ is dense in $H^1_X$). 
We now recall the definition of the Friedrich's extension of $P$ (still denoted by $P$) with domain given by,
\begin{equation}\label{domP}
 D(P) = \{u \in H^1_X(\xR^d) : Pu \in L^2(\xR^d)\}.
 \end{equation}
 It is well known that the operator $P$ is non negative and selfadjoint (we refer to Section~\ref{sec.2} for more about Friedrich's extension and the solutions to wave equations)
 
 The functional calculus allows to define for any $U_0 =(u_0, u_1) \in H^1_X \times L^2$ 
$ u(t)= \cos (\sqrt{P}) u_0 + \frac{ \sin (\sqrt{P})}{\sqrt{P} } u_1$,
the solution to the wave equation with initial data $(u_0, u_1)$.  Or first result (sharp domain of dependence) is 

\begin{theo}\label{influ}
 Let $U_0= (u_0, u_1) \in  H^1_X \times L^2$ and consider the solution to the subelliptic wave equation $U= (u, \partial_t u)= S(t) (U_0) $ given by the functional calculus.  Let $d_P$ denote the sub-riemanian distance associated to the operator $P$ (see Section~\ref{sec.3} for more details), and 
 for $x_0 \in \xR^d, t_0>0, \rho>0 $,
 \begin{equation}\label{cone}
 \begin{aligned}
  &C(t_0) = \{(t,x) \in [0, t_0) \times \xR^{d}: d_P(x,x_0) < t_0-t\},\\
  &B_P(x_0,\rho) = \{ x \in  \xR^d: d_P(x,x_0) < \rho\}.
  \end{aligned}
  \end{equation}
 Assume that $u_0 =u_1 = 0$ in $B_P(x_0,t_0) $. Then   $u$  vanishes within the cone $C(t_0)$.
\end{theo}

Unique continuation for   square roots of   Laplacian with respect to a $C^\infty$ metric has been obtained by  Masuda \cite{Ma}.  For more general fractional powers of elliptic operators  see for instance, \cite{FF}, \cite{Ru}, \cite{Y} and \cite{GSU}.  
 As an application we prove,
 \begin{theo}\label{main}
Assume \eqref{lie} and that the coefficients of the vector fields $X_j$ are real analytic. Then, if $u_0 \in H_X^1(\xR^d)$  is such that $u_0 = \sqrt{P} u_0 = 0$ on an open  subset $\omega \subset  \xR^d $ then $u_0 = 0$ on $\xR^d$.
 \end{theo}
Eventually we have the following more general result.
\begin{theo}\label{ucHo-s}
 Assume \eqref{lie} and that the coefficients of the vector fields $X_j$ are real analytic. Let $0<s<1$. 
 If $u_0\in D( P ^s) $ is such that $u_0 =P^s u_0 =0$ vanish on an open  subset $\omega \subset  \xR^d $ then $u_0 = 0$ on $\xR^d$.
\end{theo}

 Notice that, concerning the unique continuation,   for our operators,  the   case of $C^\infty$ vector fields remain an open question due to the lack of stable uniqueness for the operators "sum  of squares"(see~\cite{Ba}).

The plan of this note is the following.

 In the following section we first recall some basic facts about Friedrich' extension for sum of squares of vector fields and set up the framework allowing to solve wave equations (Hille Yosida's theorem). We then prove the sharp domain of dependence  property for solutions to subelliptic wave equations for sums of squares of vector fields satisfying H\"ormander bracket condition in Section~\ref{sec.3}.  In the next  section we prove the unique continuation theorem for the square root of H\"ormander's operators, relying on a strategy by Masuda and a uniqueness result by Bony (this is for this last result that we need the analyticity assumption). In the last section we extend the result on unique continuation for more general fractional powers relying on works by Chamorro-Jarrin,    Baouendi-Goulaouic and again Bony.
 
 {\bf Acknowledgements}  The research of the first author has received funding from the European Research Council (ERC) under the
European Union's Horizon 2020
research and innovation programme (Grant agreement 101097172 - GEOEDP). We would also like to thank C. Letrouit for enlightenments about R. Melrose's work~\cite{Mel}
 \section{Friedrich's extension, Hille-Yosida}\label{sec.2}
\subsection{Friedrich's extension}Let $X= (X_1, \ldots, X_r)$ be a set of real vector fields  with $C^\infty$ coefficients on   $\xR^d$ satisfying condition \eqref{lie}. 
We shall consider the operator,
\begin{equation}\label{opHo}
P = \sum_{j=1}^r X_j^*X_j
\end{equation}
where $X_j^*$ is the adjoint of $X_j$, with domain $C_0^\infty(\xR^d) \subset L^2(\xR^d)$.

This operator being symmetric and positive on his domain, we shall still denote by $P$ its Friedrichs self adjoint extension. 
\begin{defi}\label{H1X}
We  shall denote by  $H^1_X(\xR^d)$  the closure of  $C_0^\infty(\xR^d)$ for the norm,
$$    \Vert v\Vert^2_1 = \sum_{j=1}^r \Vert X_j v \Vert^2_{L^2(\xR^d)} + \Vert u \Vert^2_{L^2(\xR^d)}.$$
\end{defi}
Then the domain of $P$ is given by,
\begin{equation}
 D(P) = \{u \in H^1_X(\xR^d) : Pu \in L^2(\xR^d)\}.
 \end{equation}
 where, here, $P$ acts in the distribution sense.
 
Since $C_0^\infty(\xR^d)$ is, by definition, dense in $H^1_X(\xR^d)$ we may define,
$$H^{-1}_X(\xR^d) = \big(H^1_X(\xR^d)\big)', $$
the dual of $H^1_X(\xR^d)$ as a space of distributions.

It is then easy to see that $P$ sends $H^1_X(\xR^d)$ to $ H^{-1}_X(\xR^d)$ and that,
\begin{equation}\label{est-H-1}
 \Vert Pu \Vert_{H^{-1}_X(\xR^d)} \leq \Vert u\Vert_{\overline{H}^1_X},  \quad \forall u \in H^1_X(\xR^d).
 \end{equation}
 \subsection{The Hille Yosida theorem}
As in the sequel, we shall work with several operators with different domains, and in order to navigate between wave equations defined on these different domains, we need to have a unified framework to study these wave equations. So, rather than using the functional calculus for self adjoint operators, we shall  use the more flexible Hille Yosida theory below. Notice that (see Lemma~\ref{unic}) the functional calculus and Hille-Yosida theory give the same wave solutions (when the operator is self adjoint). 

Let us recall how Hille-Yosida theory applies to solve the Cauchy problem for the wave equations with respect the subelliptic operator $P$, 
\begin{equation}\label{CPW}
\begin{aligned}
&\partial_t^2 u  + Pu = 0  \quad \text { in } \xR \times \xR^d,\\
&u\arrowvert_{t=0} = u_0, \quad \partial_t u \arrowvert_{t=0} = u_1.
\end{aligned}
\end{equation}

 We consider the two operators, 
$$\mathcal{P} =\begin{pmatrix}0 & 1 \\ -P & 0 \end{pmatrix}, \quad \widetilde{\mathcal{P}} =\begin{pmatrix}0 & 1 \\ -P & 0 \end{pmatrix}, $$
acting respectively on, 
\begin{equation}\label{lesH}
 \mathcal{H} =H^1_X \times L^2,\quad \widetilde{\mathcal{H}} = L^2\times H^{-1}_X,  
 \end{equation}
with domains,
\begin{equation}\label{lesdom}
 D(\mathcal{P} )=  D(P)\times H^1_X, \quad D(\widetilde{\mathcal{P}} )= H^1_X\times L^2,
 \end{equation}
where $D(P)$ has been defined in \eqref{domP}.
Recall that these domains are actually characterized as follows:
$$ D(\mathcal{P} )= \{ U \in \mathcal{H} : \mathcal{P}_0 U \in \mathcal{H}\}, \quad  
   D(\widetilde{\mathcal{P} })= \{ U \in \widetilde{\mathcal{H}} : \mathcal{P}_0 U \in \widetilde{\mathcal{H}}\},$$ 
where $\mathcal{P}_0$ is the operator $\begin{pmatrix}0 & 1 \\ -P & 0 \end{pmatrix}$ acting on distributions. 

Then setting $U = (u,\partial_t u), U_0 = (u_0,u_1)$,  the problem \eqref{CPW} is equivalent to,
\begin{equation}\label{syst}
\partial_t U  + \mathcal{P}U = 0  \quad \text { in } \xR \times \xR^d,\quad U\arrowvert_{t=0} = U_0.
\end{equation}

Recall that an operator $T$, with domain $D(T)$ defined on a Hilbert $H$ is maximal dissipative (resp. accretive) if it satisfies, 
 $$  \Re (TU,U)_H\leq 0\,  \quad (\text{resp. } \geq 0), \quad  \forall U \in D(T),$$
 $$ \exists \lambda <0 \, \, (\text{resp. }   \lambda >0):  \forall U \in H \quad \exists V \in D(T) :  (T+ \lambda)V=U.$$

It is easy to see that the operators $\mathcal{P} -\text{Id}$ and $ \widetilde{\mathcal{P}} -\text{Id}$ are maximal dissipative (resp. $\mathcal{P} +\text{Id}$ and $ \widetilde{\mathcal{P}} +\text{Id}$ are maximal accretive).

 Consequently 
$\mathcal{P}$ and $ \widetilde{\mathcal{P}}$ are the generators of  strongly continuous groups of operators on $\mathcal{H}$,  and $\widetilde{\mathcal{H}}$, $S(t)$ and $\widetilde{S}(t)$. 
\begin{theo}[\cite{BuGe}, Th\'eor\`eme 2.10 and Proposition 2.24]  \label{BuGe}
Consider  a strongly continuous group of operators $\Sigma(t)$ on a Hilbert space $H$, and $A$ its infinitesimal generator with domain $D(A)$. Then, 
\begin{enumerate}
\item For any \,$ U_0\in D(A)$, the function $U: t\mapsto U(t) =\Sigma(t) U_0$  is the unique function in $  C^1 (\xR,  H) \cap  C^0 (\xR,  D(A))$ such that, 
$$ U(0) = U_0, \qquad \frac d {dt } U(t) = A U(t).$$
\item 
For any  $U_0 \in H$, the fonction $ U: t\mapsto U(t) =\Sigma(t) U_0$ is the unique    function  in $C^0(\xR,  H)$ such that,
 
$(i)$  $U(0) = U_0,$

$(ii)$ for any $\psi \in C^1_0 ( \xR)$ we have, 
$$ \int \psi(t) U(t) dt \in D( A), \text{ and } A \Bigl( \int \psi(t) U(t) dt \Bigr) = -  \int \psi'(t) U(t) dt, $$
(in other words, $U$ is a distribution in time with values in the domain of $A$ and it satisfies the equation $\partial_t U = A U $ as a distribution in time).
\end{enumerate}
\end{theo}
Let us now check that when applying Theorem~\ref{BuGe} to $\mathcal{P}$, we can weaken the continuity assumption on $U$ for the uniqueness part in $(2)$, and assume only $U\in C^0(\xR, \widetilde{\mathcal{H}})$, while still assuming $U_0\in \mathcal{H}$.  Indeed, this follows directly from the following observation.

\begin{lemm}\label{unique2}
Let $U_0\in \mathcal{H}$.  According to \eqref{lesH} and \eqref{lesdom} we have,   $\mathcal{H} = D(\widetilde{\mathcal{P}})$. Applying (1) in Theorem \ref{BuGe} with $H = \widetilde{\mathcal{H}}, \, A = \widetilde{\mathcal{P}}$ and  (2) with $   H = \mathcal{H},\,  A = \mathcal{P}$ we obtain two solutions $\widetilde{U} (t)$ and $U(t)$. Then,
 these two solutions coincide i.e. we have $ \widetilde{U} (t) = U(t)$.
\end{lemm}
\begin{proof}
Indeed, since $U_0\in D(\widetilde{\mathcal{P}})$, we have, 
$$ \widetilde{U} (t) \in C^1(\xR,  \widetilde{\mathcal{H}}) \cap C^0(\xR,   \mathcal{H}), $$
and it satisfies,
$$ \frac d {dt} \widetilde{U} (t) = \widetilde{\mathcal{P}} \widetilde{U}(t), \quad  \widetilde{U}(0) = U_0.$$  
This implies, 
$$  \widetilde{\mathcal{P}}\Big( \int \psi(t)\widetilde{U}(t) dt\Big) = \int \psi(t) \widetilde{U}'(t) dt = - \int \psi'(t) \widetilde{U}(t) dt,$$
where the equality holds in $\widetilde{\mathcal{H}}$. We deduce that in the distribution sense,
$$  {\mathcal{P}}_0\Big( \int \psi(t)\widetilde{U}(t) dt\Big) = \int \psi(t) \widetilde{U}'(t) dt  \in \mathcal{H} ,$$ which also implies by definition that, 
$$\int \psi(t)\widetilde{U}(t) dt\in D(\mathcal{P}).$$ This implies by the point (2) in Theorem~\ref{BuGe} that $\widetilde{U}(t)= S(t) U_0 = U(t)$.
\end{proof}
We end this section with an elementary lemma showing that Hille Yosida theory and the functional calculus of self adjoint operators allow to define the same solutions to wave equations
\begin{lemm}\label{unic}
Let $Q$ non negative  defined on $L^2( \mathbb{R}^d)$ with domain $D(Q)$ be self adjoint. Then the two solutions of wave equations given respectively by Hille-Yosida Theorem 
$$ U(t) U_0= S_A(t), A= \begin{pmatrix} 0 & 1\\ - Q & 0\end{pmatrix},$$ with 
$A$ defined on $D(Q^{1/2}) \times L^2$ with domain $D(Q) \times D(Q^{1/2})$ and the functional calculus of self adjoint operators
$$ V(t) = \begin{pmatrix} u(t), \partial_t u(t) \end{pmatrix}, \qquad u(t) = \cos( t\sqrt{Q}) u_0 + \frac{ \sin(t\sqrt{Q})} {\sqrt{p}} u_1$$ 
are equal.
\end{lemm}
\begin{proof}
It is clear that $t\mapsto V(t) \in C( \mathbb{R}; D(Q^{1/2}) \times L^2)$ satisfies conditions i) in Theorem~\ref{BuGe} and to conclude, it is enough to prove  that it satisfies also  conditions ii) in Theorem~\ref{BuGe}. 
We have since $u_0\in D(Q^{1/2})$ and $u_1 \in L^2$,
\begin{multline}\label{int1}
\int \psi(t) u(t) dt =  \int \psi(t) \cos( t\sqrt{Q}) u_0 + \frac{ \sin( t\sqrt{Q})}{\sqrt{Q}}  u_1 dt \\
=  \int \psi(t)  \frac{d} {dt} \bigl( \frac{\sin( t\sqrt{Q}) }{\sqrt{Q}} u_0 - \frac{\cos( t\sqrt{Q}) }{{Q}} u_1 \bigr) dt 
 = - \int \psi' (t)  \bigl( \frac{\sin( t\sqrt{Q}) }{\sqrt{Q}} u_0 - \frac{\cos( t\sqrt{Q}) }{{Q}} u_1 \bigr) dt \in D(Q). 
 \end{multline}
Similarly 
\begin{multline}\label{int2}
 \int \psi(t) \partial_t u(t) dt = \int \psi(t) { \cos( t\sqrt{Q})} u_1 -\sqrt{Q} \sin( t\sqrt{Q}) u_0  dt \\
 =  \int \psi(t)  \frac{d} {dt} \bigl( {\cos( t\sqrt{Q}) } u_0 +  \frac{\sin( t\sqrt{Q}) }{\sqrt{Q}} u_1 \bigr) dt 
 = - \int \psi' (t)  \bigl(   {\cos( t\sqrt{Q}) } u_0 \frac{\sin( t\sqrt{Q}) }{\sqrt{Q}} u_1\bigr) dt \in D(Q^{1/2}). 
 \end{multline}
 and from~\eqref{int1}, \eqref{int2} we deduce 
 $$ \begin{pmatrix} 0 & 1\\ - Q & 0\end{pmatrix} \int \psi(t) \begin{pmatrix} u(t)\\ \partial_t u(t) \end{pmatrix}  dt = - \int \psi'(t) \begin{pmatrix} u(t)\\ \partial_tu(t) \end{pmatrix} dt 
 $$ 
 \end{proof}
 \section{Domains of dependence}
\label{sec.3}
\subsection{The metrics}
 We shall denote by $d_P$ the sub-Riemanian  distance attached to the operator $P$  as defined in \cite{J-SC} Proposition 3.1. It is defined  as follows. If $x,y \in \xR^d$  we set,
 \begin{align*}
  \mathcal{C}(x,y) = \{\sigma:[0,1]\to& \xR^d, \,  \sigma \text{ is a Lipschitz curve}, \\
  &\sigma(0) = x, \, \sigma(1) = y \text{ and }, 
    \dot{\sigma}(t) = \Sigma_{j=1}^r a_j(t) X_j(\sigma(t)) \text{ for a.e. } t.\}, 
  \end{align*}
 Then,
 \begin{equation}\label{dist-se}
  d_P(x,y) = \text{ inf } \Big\{  \Big(\sum_{j=1}^r \int_0^1 \vert a_j(t)\vert^2\, dt\Big)^\mez, \quad \sigma \in \mathcal{C}(x,y)\Big\}.
  \end{equation}
 Notice that if we denote by $\Delta_0$ the flat Laplacian in $\xR^d$ then $d_{\Delta_0}(x,y) = \vert x-y\vert$ (the Euclidian distance).
 \subsection{Domains of dependence}
 The main result of this section  concerns the domain of dependence of the solution  of  the subelliptic wave problem,
 \begin{equation}\label{sewe}
 \begin{aligned}
 &(\partial_t^2 + P ) u = 0 \text{ in } \xR\times \xR^d,\\
 &u \arrowvert_{t=0} = u_0, \quad \partial_t u \arrowvert_{t=0} = u_1.
 \end{aligned}
 \end{equation}

We first begin with a corollary of Theorem~\ref{influ}
 
\begin{coro}\label{coro1}
 Let $U_0= (u_0, u_1) \in (H^1\cap  H^1_X) \times L^2$ and consider the solution to the subelliptic wave equation $U= (u, \partial_t u)= S(t) U_0$ given by the Hille-Yosida theorem. Assume that $u_0 =u_1 = 0$ in a ball $\{x \in \xR^d: \vert x-x_0\vert<r_0\}$.  Then there exists $\eta>0$ and $r>0$ such that $u= 0$ in the set $\{(t,x): 0\leq t< \eta, \vert x-x_0\vert< r\}$.
\end{coro}
\begin{proof}[Proof of the Corollary]
Indeed, notice that for every compact $K \subset \xR^d$ there exist $M>0, \rho_0>0$ such that  for every $x_0\in K$   we have $B_P(x_0, \rho) \subset B_{\Delta_0}(x_0, M\rho)$ for every $0<\rho<\rho_0$.  
 
 On the other side,  according to \cite{F-P},  the hypothesis \eqref{lie} implies that for every compact $K \subset \xR^d$ there exist  $M>0, \rho_0>0, \delta>0$ such that $B_{\Delta_0}(x_0, \rho) \subset B_P(x_0, M \rho^\delta)$ for every $x_0 \in K$.
 
 Then Corollary \ref{coro1} follows easily from Theorem \ref{influ}.
 \end{proof}
 \begin{remas}
 \ 
 
\begin{enumerate}
\item  In the case where $\xR^d$ is replaced by a compact manifold Theorem \ref{influ} appears in a work of Melrose~\cite{Mel}. Our proof below applies also to this setting and hence gives a fairly self contained (and different) new proof.

\item In the case where the vector fields $(X_1, \ldots, X_r)$ are of Heisenberg type, Theorem \ref{influ} has been proved by the second author in \cite{Zu}.
\end{enumerate}
\end{remas}
\begin{proof}[Proof of Theorem~\ref{influ}]
To prove this result, the idea is to solve first the (elliptic) wave equation associated to the family of operators $P_k = P - \eps_k \Delta_0$, where $\Delta_0$ is the flat Laplacian, with initial data a sequence  of smooth initial data $(u_{0,k} , u_{1,k}) $ converging to $(u_0, u_1) $ in $H^1_X\times L^2$, with a proper choice of $ \eps_k$. Setting $X_ k = (X_1, \ldots, X_r, \sqrt{ \eps_k} \partial_1, \ldots, \sqrt{ \eps_k} \partial_d)$ we notice that $H^1_{X_ k} = H^1 \cap  H^1_X$ and hence it is possible to define such solution. Then we conclude by using the standard dependency domain results for elliptic wave equations and a limiting process. 

In all that follows we take  $(u_0,u_1) \in H^1_X(\xR^d) \times L^2(\xR^d)$ such that,
\begin{equation}\label{u0u1}
 u_0 =  u_1 = 0 \text{ in } B_P(x_0,t_0).
 \end{equation}

We start with the choice of sequences $(u_{0,k} , u_{1,k}) $ and $\eps_k$. From the definition of $H^1_X$, there exists $u_{0,k}\in C^\infty_0 ( \mathbb{R}^d)$ converging to $u_0$ as $k\rightarrow + \infty$ in $H^1_X$. There exists also a sequence $(u_{1,k})\subset C^\infty_0 (\mathbb{R}^d)$ converging to $u_1$ in $L^2( \mathbb{R}^d)$.  We shall set, 
\begin{equation}\label{choix}
 \eps_k = k^{-1} \bigl(1+ \max_{1\leq j\leq k} \| u_{0,j}\|_{H^1}^2\bigr) ^{-2}, \quad P_k = P - \eps_k \Delta_0.
 \end{equation}
Notice that the sequence $\eps_k$ is nonincreasing and
$$ \eps_k \|  u_{0,k}\|_{H^1}^2 \rightarrow 0 \text{ when } k \rightarrow + \infty.$$
 Let $d_{P_{k}}$ be the Riemanian metric attached to this operator, defined as in \eqref{dist-se}  and introduce the sets, 
\begin{equation}\label{conek}
 \begin{aligned}
  &C_k(t_0) = \{(t,x) \in [0, t_0) \times \xR^{d}: d_{P_k}(x,x_0) < t_0-t\},\\
  &B_{P_k}(x_0,\rho) = \{ x \in  \xR^d: d_{P_k}(x,x_0) < \rho\}.
  \end{aligned}
  \end{equation}
Then we have,
\begin{prop} (\cite{J-SC}, Proposition 3.1)\label{prop-limite}
Let $x_0 \in \xR^d$ then for every $x  \in \xR^d$ we have,
\begin{enumerate}
\item The squence $k \mapsto d_{P_k}(x, x_0)$ is non decreasing and bounded by $d_{P}(x, x_0)$
\item $\lim_{k \to +\infty}  d_{P_{k}}(x,x_0) =  d_P (x,x_0)$.
\item  The convergence is uniform on every compact.
\end{enumerate}
\end{prop}
\begin{proof} The non decreasing property and the pointwise convergence have been proved in \cite{J-SC}.   The uniformity is a consequence of Dini's theorem since   the sequence $k \mapsto d_{P_{k}}(x,x_0) $  is non decreasing  and the functions $x \mapsto  d_{P_{k}}(x,x_0)$ and  $x \mapsto  d_P (x,x_0)$ are continuous.
\end{proof}
Now fix $\delta>0$ small. By Proposition \ref{prop-limite} on can find $k_0\geq 1$ such that for $k \geq k_0$ and for $x$ in a large fixed ball around $x_0$,
$$ d_P(x,x_0) - \frac{\delta}{2} \leq d_{P_k}(x,x_0) \leq d_{P}(x,x_0).$$ 
It follows that,
\begin{equation}\label{inclu}
B_{P_k}(x_0, t_0- \delta) \subset B_P(x_0, t_0 - \frac{\delta}{2}),  \quad  B_{P_k}(x_0, t_0 - \delta ) \subset B_{P_k}(x_0, t_0 - \frac{\delta}{2} )\subset B_P(x_0, t_0 ).
\end{equation}

Let $\chi \in C^\infty(\xR^d)$ be such that, 
$$\chi(x) = 0  \text{ in }B_P\big( x_0, t_0- \frac{\delta}{2}\big), \quad \chi(x) = 1 \text{  in }  B_P(x_0, t_0)^c$$
and set,
$$v_{j,k}(x) = \chi(x) u_{j,k}(x), \quad j=0,1.$$
Then, according to \eqref{inclu}, 
\begin{equation}\label{vjk}
\begin{aligned}
  &v_{0,k} \in C_0^\infty(\xR^d),\quad v_{0,k} = 0 \text{ in } B_{P_k}(x_0, t_0-\delta) \quad \text{ and } (v_{0,k}) \to u_0 \text{ in } H^1_X(\xR^d).\\
 &v_{1,k} \in C_0^\infty(\xR^d),\quad v_{1,k} = 0 \text{ in } B_{P_k}(x_0, t_0-\delta) \quad \text{ and } (v_{1,k}) \to u_1 \text{ in } L^2(\xR^d).
 \end{aligned}
  \end{equation}
Now we set,
\begin{equation}\label{U0k}
U^0_{k} = (v_{0,k},v_{1,k}).
\end{equation}

 For the limiting process the key point is the following result.
 
  Let us set, with $ \eps_k \in (0,1) $  defined in \eqref{choix},
$$\mathcal{P}_ {k} =\begin{pmatrix}0 & 1 \\ -P +  \eps_k \Delta_0& 0 \end{pmatrix}.   $$
 \begin{prop}
 For $U_0= (u_{0}, u_1) \in \mathcal{H}:= (H^1\cap  H^1_X) \times L^2$,  consider  the solution,
 $$ U_k (t) = S_{k} (t) U^0_{k}  \in C^0(\xR, \mathcal{H})$$ 
 given by Theorem \ref{BuGe} with the operator $\mathcal{P}_k$ (with initial data $U^0_{k})$. Then, when $k \rightarrow \infty$, the sequence $(U_k)$ converges weakly to, 
 $$ U(t) = S(t) U_0,$$
 the solution to the wave equation given by the same theorem with $\mathcal{P}$ and initial data $U_0= (u_0,u_1)$. 
 \end{prop}
 \begin{proof}
 We start with the energy conservation, 
 \begin{align*}
 & E_k (t) = \int \big(\eps_k |\nabla _x u_k|^2 + \sum_{j = 1}^r |X_j  u_k |^2  + |\partial_t u_k|^2 \big)(t,x)dx = E_k (0),\\
 & E_k(0) = \int \big(\eps_k  \vert  \nabla _x v_{0,k} \vert^2 + \sum_{j = 1}^r \vert X_j  v_{0,k} \vert ^2  + \vert v_{1,k}\vert^2 \big)(x)dx.
 \end{align*}
 Moreover, writing $u_k(t, \cdot) = u_k(0, \cdot) + \int_0^t\partial_s u_k(s, \cdot)\, ds$,  we obtain,
 $$ \| u_k (t) \|_{L^2} \leq \Vert u_0\Vert_{L^2} +  |t| E_k(0)^{1/2}.$$

 From the choice of $\eps_k$ given in~\eqref{choix},  we see easily that, 
 $$\exists \, C>0 :\quad E_k(0) \leq  C, \quad \forall k.$$
 
 Since $E_k$ controls uniformly the square of the norm in $ H^1_X \times L^2$, we deduce that we can find a subsequence $U_{k_n}$ such that 
 \begin{equation}\label{cv1}
  U_{k_n} \rightarrow U  \text{ weakly in }  \,  L^\infty (\xR,  H^1_X \times L^2), \text{ as } n \rightarrow + \infty
 \end{equation}
 
 From the equation (satisfied in the distribution sense),
 \begin{equation}\label{eqUeps}
   \partial_t  U_{k_n} = \begin{pmatrix} 0 & 1 \\ -P + {\eps_{k_n}} \Delta_0 & 0 \end{pmatrix} U_{k_n},
   \end{equation}
 we deduce first that, 
 \begin{equation}\label{eq}
 \partial_t  U = \begin{pmatrix} 0 & 1 \\ -P  & 0 \end{pmatrix} U, \quad \text{ in } \mathcal{D}'(\xR\times \xR^d).
 \end{equation}
 
Now, using \eqref{est-H-1} and \eqref{eqUeps}, we see that,
\begin{equation}\label{borne} 
 \partial_t  U_{k} \text{  is bounded in } L^\infty  ( (-1,1),  L^2\times (H^{-1} +H^{-1}_X)).
 \end{equation}

Let $\Phi = (\varphi, \psi) \in (C_0^\infty(\xR^d))^2$ and set $\theta_k(t) = \langle U_k(t), \Phi\rangle$ where $\langle \cdot, \cdot\rangle$ is the distribution duality.  Then $(\theta_{k}) \subset C^0([-1,1])$ and it follows from \eqref{cv1} that $(\theta_{k_n})$ converges to $ \langle U(\cdot), \Phi\rangle$ in $\mathcal{D}'(\xR)$.  

On the other hand we have $\partial_t \theta_{k_n}(t) = \langle \partial_t U_{k_n}(t), \Phi\rangle$  so by \eqref{borne} $(\partial_t \theta_{k_n})$ is uniformly bounded in $L^\infty([-1,1])$. Therefore the set $(\theta_{k_n})$ is equicontinuous, and hence from from the Ascoli theorem it is relatively compact  in $C^0([-1,1])$.   But there exist only one possible accumulation point $ \langle U(\cdot), \Phi\rangle$ which implies  that $\theta_{k_n}$ converges to $\theta $ in $C^0([-1,1])$. As a consequence, we have $\theta \in C^0$ and 
$$U(0)= \lim_{n\rightarrow + \infty}  = U_{0,k_n}= U_0, \text{ in } \mathcal{D}'(\xR^d).$$ 

  We also have from the characterization of Hille-Yosida solutions given in Theorem~\ref{BuGe}, (writing $k$ instead of $k_n$)
 $$ \int \psi(t) U_{k} (t) dt \in D( \mathcal{P}_{k}), $$
 \begin{equation}\label{HY}
  \mathcal{P}_{k} \Bigl(\int \psi(t) U_{k}  (t) dt \Bigr) = - \int \psi'(t) U_{k}  (t) dt .
  \end{equation}
We can now pass to the limit $k\rightarrow + \infty$ in ~\eqref{HY} and get (in distribution sense),
\begin{equation}\label{bef}
 \mathcal{P} \Bigl( \int \psi(t) U (t) dt \Bigr) = - \int \psi'(t) U(t) dt.
 \end{equation}
Now, clearly $\int \psi'(t) U (t) dt \in \mathcal{H}$ and from~\eqref{bef} we deduce,  
$$ \mathcal{P} \Bigl( \int \psi(t) U (t) dt \Bigr) \in \mathcal{H} \text{ so } \int \psi(t) U (t) dt \in D(\mathcal{P}).$$ 
From~\eqref{eq}, we know that, 
$$\partial_t U \in L^\infty(\xR,  L^2\times H^{-1}_X),$$ and consequently, 
$$U \in C^0(\xR,   L^2\times H^{-1}_X) $$
while on the other hand  we have $U(0) = U_0$ (in $\mathcal{D}'$ but hence also in $\mathcal{H}$).  Now the uniqueness part in Theorem~\ref{BuGe} for the operator $\widetilde{\mathcal{P}}$ implies  $U(t) = \widetilde{S}(t) U_0$, and Lemma~\ref{unique2} imply in turn $U(t) = {S}(t) U_0$. Eventually, since the limit is unique, the extraction process is unnecessary and the whole family $(U_k)$ converges to $U= S(t) U_0$.
\end{proof}
 
Then we have the following result.   
\begin{prop}(\cite{E} Theorem 8, chapter 7) \label{cas-ell}
Let $(u_0, u_1) \in H_X^1(\xR^d)\times L^2(\xR^d)$ and  let $u_{k}$ be the   solution of the problem,
$$( \partial_t^2  + P_{k}) u_{k}= 0 \text { in } \xR^{1+d}, \quad u_{k}\arrowvert_{t=0} = v_{0,k}, \quad \partial_t u_{k}\arrowvert_{t=0} = v_{1,k}$$
given by Hille Yosida Theorem, where $(v_{j,k})$ for $j=0,1$ satisfy \eqref{vjk}.

 Then   $u_{k}$  vanishes within $C_k(t_0-\delta)$.
\end{prop}
 
Let us now come back to the proof of Theorem~\ref{influ}

Assume that $(u_0, u_1)$ vanish in $B_P(x_0,t_0)$. Then, by \eqref{vjk}   for $k$ large enough, $(v_{0,k}, v_{1,k})$ vanish in $B_{P_k}(x_0, t_0-\delta)$. Then  Proposition \ref{cas-ell}   implies that $u_{k}$ vanish in $C_k(t_0-\delta) = \{(t,x) \in [0, t_0-\delta)\times \xR^d: d_{P_{k}}(x,x_0) <t_0-\delta-t\}$, thus in the set  $C(t_0-\delta)= \{(t,x) \in [0, t_0-\delta)\times \xR^d: d_{P}(x,x_0) <t_0-\delta-t\}$ since $d_{P_{k}} \leq d_P$. Since $(u_{k})$ converges to $u$, in the space of distributions, we deduce that $u$ vanishes in the set $C(t_0-\delta)$.  Since this holds for every $\delta>0$ we deduce that  $u$ vanishes in the set $C(t_0)$.

The proof of Theorem \ref{influ} is complete.
\end{proof}

\section{Unique continuation of square roots of H\" ormander's operators}
In this section we prove Theorem~\ref{main}. We assume in this section that the vector fields $X= (X_1, \ldots, X_r)$ have analytic coefficients on   $\xR^d$.  

Since the operator $P$ is a positive Friedrichs self adjoint extension, we can define its square root $\sqrt{P}$ with domain $H^1_X(\xR^d)$ introduced in Definition \ref{H1X}.
 The proof of Theorem~\ref{main} combines three arguments. First,   our result on the   dependency  domain
for the solutions of the weakly hyperbolic operator $Q= \partial_t^2 + P$, then an argument of holomorphic extension due to Masuda \cite{Ma},  and eventually a result of Bony \cite{B}, concerning the unique continuation for $P$.
   
 \subsubsection{Masuda's argument}
Let $u_0 \in  H^1(\xR^d)\cap H_X^1(\xR^d)$ be such that   $u_0 = \sqrt{P} u_0 = 0$ on an open  subset $\omega \subset \xR^d$. Consider the problem,
\begin{equation}\label{demi-w}
(i \partial_t + \sqrt{P})u = 0, \quad u \arrowvert_{t=0} = u_0.
\end{equation}
Since $ \partial_t u\arrowvert_{t=0} = i \sqrt{P}u_0$ it follows that  $u$ is also a solution  of the problem \eqref{sewe} with $u_0 = u_1 = 0$ in $\omega$. Therefore, by Corollary \ref{coro1}  there exists $\delta>0$ and $\omega_1 $ open with $\overline{\omega_1} \subset \omega$ such that,
\begin{equation}\label{u=0}
 u  = 0 \text{ in }  (-\delta, \delta)\times \omega_1.
 \end{equation}
Now since $\sqrt{P} $ is a positive self adjoint operator in $L^2(\xR^d)$ by the spectral theorem the solution of \eqref{demi-w} can be written as,
\begin{equation}\label{demi-w2}
  u(t,\cdot) = e^{it \sqrt{P}} u_0= \int_0^{+\infty} e^{it\lambda} dE(\lambda)u_0.
  \end{equation}
  The above formula shows that $u$ has a holomorphic extension to the upper half plane $\text{Im }z >0$ with values in $L^2(\xR^d)$ in the sense that the function,
  $$z \mapsto  \int_0^{+\infty} e^{iz\lambda} \langle dE(\lambda)u_0, \varphi\rangle$$
  is holomorphic  $\text{Im }z >0$ for any $\varphi \in L^2(\xR^d)$.

 We shall still denote by $u(z, \cdot)$ this holomorphic extension.
  
  According to \eqref{sewe} this extension satisfies,
  \begin{equation}\label{cauchyC}
  \partial_z^2u + Pu = 0, \quad \text{ in } \{z \in \xC: \text {Im } z >0, x \in \xR^d\}.
  \end{equation}

  Now for $z \in \xC$  such that $\vert \text{Re } z \vert < \delta$ we set,
  \begin{equation}\label{U}
   U(z,\cdot) =
   \begin{cases}
    u(z,\cdot), & \text{Im }z \geq0,\\ \overline{u}(\overline{z},\cdot), &\text{Im } z <0. 
    \end{cases}
   \end{equation}
  Let $\varphi \in C_0^\infty(\omega_1)$ and set,
  \begin{equation}\label{theta}
  \theta(z) = \langle U( z,\cdot), \varphi\rangle
  \end{equation}

  Then  $\theta$ is holomorphic for $\text{Im } z >0$, for $\text{Im } z <0$ and continuous  in   $\{z:  \vert \text{Re } z \vert < \delta\}$ by \eqref{u=0}. By Morera's theorem,  $\theta$ is holomorphic in the set $\{z: \vert \text{Re } z \vert < \delta  \}$. Since, by \eqref{u=0}, it vanishes in the set  $\{z: \vert \text{Re } z \vert < \delta, \, \text{Im }z = 0 \}$  we have $\theta(z)= 0$ in  $\{ z: \vert \text{Re } z \vert < \delta\}$. 
  
  According to \eqref{U} and \eqref{theta} it follows that $u(z, \cdot) = 0$ for  $ \vert \text{Re }z \vert < \delta$ in $\mathcal{D}'(\omega_1)$ and consequently  $u(z,x) =0$ for almost all $x$ in $ \omega_1$.

  Now, for $x \in \xR^d, $   $ \vert \xi \vert < \delta$   let us set,
  \begin{equation}\label{u=}
  v(\xi, \eta, x) = u(\xi + i \eta,x). 
\end{equation}
  Then, by the above argument and \eqref{u=0} we have,
  \begin{equation}\label{v=0}
  v(\xi, \eta, x) = 0 \quad \text {if } \vert \xi\vert <\delta, \quad \text{ a.a } x \in \omega_1.
  \end{equation}

  On the other hand, since $u$ is holomorphic in the set $\{z:    \text{Im } z >0  \}, $ the function $v$ is harmonic with respect to $(\xi, \eta)$ in the set $ \mathcal{O} = \{(\xi, \eta): \vert \xi \vert <\delta,\,  \eta  >0\}$. Therefore,
  $$(\partial^2_\xi + \partial^2_\eta) v(\xi, \eta, x) = 0, \quad  \text{in } \mathcal{O} \times \xR^d.$$
 Now, by \eqref{cauchyC} we have in $\mathcal{O}\times \xR^d$, 
  $$\partial^2_\eta v(\xi, \eta, x) = - \partial^2_z u(\xi + i \eta,x) = Pv(\xi,\eta, x).$$
Using the two above equations  we find that,
  \begin{equation}\label{Ho}
   ( -\partial_\xi^2 v -2 \partial_\eta^2 v + Pv )(\xi, \eta, x) = 0, \quad \text{in } \mathcal{O} \times \xR^d.
  \end{equation}
  Notice that by H\" ormander's theorem \cite{H}, $v$ is a $C^\infty$ function in $\mathcal{O} \times \xR^d$.
  \subsubsection{Bony's result}
  In \cite{B} Bony has proved the following.
  \begin{theo}\label{Bony}
  Let $(Y_1, \ldots, Y_m)$ be a set of vector fields on   $\xR^N$ with analytic coefficients   satisfying   H\"ormander's condition namely,
  $$\text{The Lie algebra  generated by these vector fields has everywhere maximal rank}.$$
  Let $w$ be a distribution solution,  in $\xR^N$,  of the equation $\sum_{k=1}^m Y_k^*Y_kw + \sum_{j=1}^m a_j(x) Y_j w +a_0(x)w= 0$ where the $a_j's$ are analytic If $w$ vanishes in any open subset  then $w$ vanishes identically in $\xR^N$.
  \end{theo}
Coming back to the equation \eqref{Ho} we set, 
$$ Y_1= \partial_\xi,\, Y_2= \sqrt{2} \partial_ \eta, \, Y_{j+2} = X_j,  \, j =1, \ldots,r .$$
Then the function $v$ is a solution of the equation,
$$\sum_{k=1}^{r+2} Y_k^*Y_k v = 0$$
 in   $ \mathcal{O}  \times \xR^d$. The system $(Y_1, \ldots, Y_{r+2})$ 
 has analytic coefficients and satisfy H\"ormander's condition in $ \mathcal{O}  \times \xR^d$.  Since, by \eqref{v=0},   $v$ vanishes on the open set $  \{(\xi, \eta,x):  \vert \xi \vert < \delta,  \,\eta>0, \, x \in  \omega_1\} $ it follows from Bony's result that,
 $$v  \text{ vanishes identically in } \{(\xi, \eta): \vert \xi \vert<\delta, \eta>0\}  \times \xR^d.$$
 
 Since $v$ is continuous in  $\{(\xi, \eta): \vert \xi\vert <\delta, \,\eta \geq  0\} $ we deduce from \eqref{u=} that $u_0(x) = v(0,0, x) = 0$ for all $x \in \xR^N$, which completes  the proof of Theorem \ref{main}.
 
 \section{Unique continuation for $s$-powers  of H\" ormander's operators, $0<s<1$.}
 
 In this section we consider, as before, a system of vector fields  $(X_1, \ldots, X_r)$ on $\xR^d$, on which we make, as in the second part,  the following assumptions,
\begin{equation*} 
\begin{aligned}
& (ii) \quad \text{the Lie algebra generated by these vector fields has maximal rank at every point in } \xR^d,\\
& (i) \quad \text{ the coefficients of the } X_j's  \text{ are analytic in }  \xR^d. 
\end{aligned}
\end{equation*}

We shall consider the operator, $P = \sum_{j=1}^r X_j^*X_j$, 
where $X_j^*$ is the adjoint of $X_j$ and denote by $H^1_X$ the closure of $C^\infty_0 ( \mathbb{R}^d)$ for the norm, 
$$ \| u \|^2_{H^1_X} = \sum _j \| X_j u \|_{L^2( \mathbb{R}^d)}^2$$
We shall  work, as before, with the Friedrich's extension of $P$ (still denoted by $P$) with domain given by,
\begin{equation}
 D(P) = \{u \in H^1_X(\xR^d) : Pu \in L^2(\xR^d)\},
 \end{equation}
 and it is well known that the operator $P$ is non negative and selfadjoint.
 
  Then for $0<s<1$ we can define, by the functional calculus, the fractional powers $P^s$ of $P$. It is defined by the formula,
  $$P^s \varphi(x) = \int_0^{+\infty} \lambda^s dE(\lambda) \varphi, $$
  where $E(\lambda) $ is the spectral decomposition of $P$ with domain,
  $$D(P^s) = \{\varphi \in L^2(\xR^d):  P^s \varphi \in L^2(\xR^d)\}.$$

The goal of this section is to prove   Theorem~\ref{ucHo-s}

\subsection{A result by Chamorro-Jarrin}

For $\psi \in L^2(\xR^d)$ we denote by $v = e^{-tP} (\psi)$ the solution of the heat problem,
$$\partial_t v + P v = 0, \quad v \arrowvert_{t=0} = \psi.$$
\begin{theo}\label{cha-ja}
Let .  Let $s\in (0,1)$ and $\varphi \in L^2(\xR^d).$ Consider the following problem in $(0, +\infty) \times \xR^d,$
\begin{equation}\label{ext-pb}
 \partial_t^2 u(t,x) + \frac{1-2s}{t} \partial_t u(t,x) - Pu(t,x) = 0, \quad u\arrowvert_{t=0} = \varphi.
 \end{equation}
Then the function $u: (0, +\infty) \times \xR^d \to \xR$ given by the formula,
\begin{equation}\label{formule}
 u(t,x) = \frac{1}{\Gamma(s)} \int_0^{+\infty} P^se^{-\tau P}(\varphi)(x) e^{-\frac{t^2}{4 \tau}} \frac{d\tau}{\tau^{1-s}},
 \end{equation}
where $\Gamma$ is the Gamma function, is a solution of \eqref{ext-pb} which is in $C^\infty(]0, +\infty[ \times \xR^d)$ and satisfies    $ t^k \partial_t^k u \in C([0,+\infty[, L^2(\xR^d))$ for every $k\in \xN$. Moreover when $\varphi$ belongs to the domain of $P^s$ we have, in the $L^2$ sense,
\begin{equation}\label{limite}
  \lim_{t\to 0^+} t^{1- 2s} \partial_t u(t, x) = C(s) P^s\varphi(x),
  \end{equation}
where $C(s) $ is a (non zero) constant  depending only on $s \in (0,1)$.

 \end{theo}
This result appears in \cite{CJ} with the extra hypothesis that $(X_1, \ldots, X_r)$ is a family of left invariant vector fields on a connected unimodular Lie group. However this theorem holds in our general case with an   almost identical proof. For the sake of completeness we  give this proof in the appendix.

\subsection{A Baouendi-Goulaouic uniqueness result}
We quote here a  particular case, adapted to our situation,  of  \cite[Theorem 4]{BG}.
\begin{theo}\label{Ba-Gou}(Baouendi-Goulaouic \cite{BG})
Let 
$$\mathcal{P} = t^2\partial_t^2 + a_1 t \partial_t + a_0 + t^2P(x, \partial_x), \quad (t,x)\in \xR \times \xR^d,$$ be a second order Fuchs type operator, where $a_1, a_2$ are real numbers and $P$ is a second order differential operator with analytic coefficients in a neighborhood of    a point  $x_0\in \xR^d$. 

Let $\lambda_1, \lambda_2$ be the two roots of the characteristic equation,
$$\lambda(\lambda-1) + a_1 \lambda + a_0 = 0,$$
and let $h \in \xN$ be such that $\mathcal{R}e (\lambda_j)< h, j = 1,2$.

Then if $v  \in C^0(\xR, \mathcal{D}'(\omega_0))$  is such that  $t^{-h}v \in C^0(\xR, \mathcal{D}'(\omega_0)),$ where $\omega_0$ is a neighborhood of $x_0$  and satisfies,
$$\mathcal{P} u = 0, \quad \partial^j_t u(0, \cdot) = 0, \quad 0 \leq j \leq h-1,$$
then $v$ vanishes identically near $(0, x_0) $ in  $\xR \times \xR^d$.
\end{theo}
\subsection{Proof of Theorem \ref{ucHo-s}}
Let $u$ be the solution of \eqref{ext-pb} given by \eqref{formule}. Recall that, $u\in C^0([0, +\infty), L^2(\xR^d))$. and by the condition in the Theorem,
\begin{equation}\label{trace0} 
 u\arrowvert_{t=0} = 0  \quad \text{ in }\omega. 
 \end{equation}
 Consequently, since  $P$  is an operator in $x$ whose  coefficients do not depend on $t,$
 \begin{equation}\label{trace0P} 
 Pu\arrowvert_{t=0} = 0  \quad \text{ in }\omega, 
 \end{equation}
 Let us set, 
 \begin{equation}\label{A=}
   A(t,x) = t^{1-2s} \partial_t u(t,x).
  \end{equation}
 Using \eqref{limite} and  the hypothesis in Theorem \ref{ucHo-s} we see that,
 \begin{equation}\label{limA}
 \lim_{t\to 0^+} A(t,x) = 0 \quad \text{ in }L^2(\omega)
 \end{equation}

\begin{equation}\label{limdA}
 \lim_{t\to 0^+} \partial_x^\alpha A(t,x) = 0 \quad \text{ in }\mathcal{D}'(\omega).
 \end{equation}
 Now according to \eqref{trace0} we can write, for $x \in \omega$,
  \begin{equation}\label{B}
  u(t,x) = t \int_0^1 (\partial_tu)(\lambda t)\, d \lambda = t^{2s} \int_0^1 \lambda^{2s-1} A(\lambda t,x)\, d \lambda: = t^{2s} B(t,x).
  \end{equation}
Since $2s-1  \in ]-1,1[$ we deduce from \eqref{limA} that,
\begin{equation}\label{limB}
 \lim_{t\to 0^+} B(t,x) = 0 \quad  \text{ in } L^2(\omega ).
 \end{equation}
 Set for $t>0, x \in \xR^d$, 
  \begin{equation}\label{v=}
  v(t,x) = t^{1-2s} u(t,x) = t B(t,x).
  \end{equation}
  then $v$ is a $C^\infty$ function in $]0, +\infty[ \times \xR^d$ which belongs to  $C^0([0, +\infty[, L^2(\xR^d))$ and satisfies,
  \begin{equation}\label{trace0bis} 
 v\arrowvert_{t=0} = 0  \quad \text{ in }L^2 (\omega). 
 \end{equation}
 Moreover,
 \begin{equation}\label{dtv}
   \partial_t v = t^{1-2s} \partial_t u  + (1-2s) t^{-2s} u = A + (1-2s) B.
   \end{equation}
 Therefore by \eqref{limA} and \eqref{limB} we deduce that,
 \begin{equation}\label{limdtv}
 \lim_{t\to 0^+} \partial_t v(t,x) = 0, \text{ in } L^2( \omega).
 \end{equation}
On the other hand, using \eqref{A=}, the equation and \eqref{B}, we can write for $t>0$,
\begin{align*}
 \partial_t A(t,x) &= t^{-2s}\big( t\partial_t^2 u + (1-2s)  \partial_t u)= -t^{-2s+1} Pu,\\
 &= t^{-2s + 1}P(t^{2s} B). = t PB.
 \end{align*}
It follows then from \eqref{limB} that,
\begin{equation}\label{limdtA}
 \lim_{t\to 0^+} \partial_t A(t,x) = 0, \text{ in } \mathcal{D}'( \omega).
\end{equation}
From \eqref{B} we get,
$$\partial_t B = \int_0^1 \lambda^{2s} \partial_t A( \lambda t, x)\, d \lambda,$$
therefore,
\begin{equation}\label{limdtB}
 \lim_{t\to 0^+} \partial_t B(t,x) = 0, \text{ in } \mathcal{D}'( \omega).
\end{equation}
Using \eqref{dtv} we deduce,
\begin{equation}\label{limdt2v}
\lim_{t \to 0^+} \partial_t^2 v(t, x) =  \lim_{t \to 0^+} \big(\partial_t A(t,x) + (1-2s)  \partial_t B(t,x)\big) = 0, \text{ in }\mathcal{D}'( \omega).
\end{equation}
Therefore,   $v$ can be extended as a function in $C^2([0, +\infty), \mathcal{D}'(\omega))$   with, 
\begin{equation}\label{vC2}
 v \arrowvert_{t=0} = \partial_t v \arrowvert_{t=0}= \partial^2_t v \arrowvert_{t=0}=0 \text{ in } \mathcal{D}'(\omega).
\end{equation}
Let us now check the equation satisfied by $v$ on $]0, +\infty[ \times \xR^d$.  Since $ u = t^{2s-1} v$ we have,
\begin{align*}
\partial_t u &= (2s-1) t^{2s-2} v + t^{2s-1} \partial_t v,\\
 \partial^2_t u &=  (2s-1) (2s-2)t^{2s-3} v + 2(2s-1)t^{2s-2} \partial_t v + t^{2s-1} \partial^2_t v,\\
 tPu &= t^{2s} Pv.
\end{align*}
It follows that,
$$ 0 = t \partial^2_t u + (1-2s) \partial_t u + t Pu = t^{2s-2}\big(t^2 \partial^2_t v + (2s-1)t \partial_t u -(2s-1) v  + t^2 Pv\big),$$
so that,
\begin{equation}\label{eqv}
t^2 \partial^2_t v + (2 s-1)t \partial_t v -(2s-1) v  + t^2 Pv =0.
\end{equation}
Setting $\widetilde{v} = H(t) v$ where $H(t) = 1 $ if $t>0,$ $H(t) = 0$ if $t<0$ we deduce from \eqref{vC2} that $\widetilde{v} $  belongs to $C^0(\xR, \mathcal{D}'(\omega))$ and is such that $t^{-2}\widetilde{v} \in     C^0(\xR, \mathcal{D}'(\omega)).$ Moreover $\widetilde{v} $  satisfies the same equation,
\begin{equation}\label{eqvtilde}
t^2 \partial^2_t \widetilde{v} + (2 s-1)t \partial_t \widetilde{v} -(2s-1) v  + t^2 P\widetilde{v} =0,  \text{ in } \mathcal{D}'(\xR\times \omega).
\end{equation}
We may apply Theorem \ref{Ba-Gou}. with $a_1 =  2s-1 , a_2 = 1-2s$. The characterisic equation is,
$$ \lambda( \lambda-1) +  (2s-1)\lambda + 1-2s = \lambda^2    + 2(s-1) \lambda + 1-2s = 0$$
whose roots are $\lambda_1 = 1, \lambda_2 = 1-2s$. We can take $h = 2$ in order that $\lambda_j <h$.

 The conclusion of Theorem \ref{Ba-Gou} is then that  there exist  $\delta >0 $ and an open set $\omega_1$ with $\overline{\omega_1} \subset \omega$ such that $ \widetilde {v}$ vanishes    in $]-\delta, \delta[ \times \omega_1$. It follows from the definition of $\widetilde{v}$ that the function $u$ vanishes identically in  $\mathcal{O}:= ]0, \delta[ \times \omega_1$.  But  in $ ]0, +\infty[ \times \xR^d$ the operator 
 $$\mathcal{P} = \partial_t^2 + \frac{1-2s}{t} \partial_t - \sum_{j=1}^r X_j^* X_j$$
 is an operator "sum of squares" with analytic coefficients and $u$ is a solution of the equation $\mathcal{P} u = 0$ in $ ]0, +\infty[ \times \xR^d$ which vanishes in $\mathcal{O}$. Bony's theorem (see Theorem \ref{Bony}) shows that $u$ vanishes identically in $ ]0, +\infty[ \times \xR^d$, therefore $ \varphi = u\arrowvert_{t=0}$ vanishes identically in $\xR^d$.
 \section{Appendix}
\subsection{Proof of Theorem \ref{cha-ja}}
As written above, the following proof is taken from  to \cite{CJ}.
Notice first that when $\varphi \in L^2(\xR^d)$ the function $e^{-\tau P} \varphi,$ for $\tau >0,$ belongs to the domain of $P^s$ (in fact to the domain of $P^\rho$ for every $\rho >0$),  so $  P^s e^{-\tau P}(\varphi)$ is well defined.  Indeed, since $\lambda^{2s} e^{-2\tau \lambda} \leq c(\tau,s) $ for every $\lambda>0$ we have,
$$\int_0^{+\infty} \lambda^{2s}e^{-2\tau \lambda}\, d\big( E(\lambda)\varphi , \varphi\big)  \leq c(\tau,s)\Vert \varphi\Vert^2_{L^2(\xR^d)}.$$
Now let us show that for $\varphi \in L^2(\xR^d)$ the function $u$ is well defined and belongs to $C^0([0, +\infty), L^2(\xR^d))$. Indeed we have,
$$u(t,x)  = \frac{1}{\Gamma(s)}\int_0^{+\infty} \lambda^s \int_0^{+\infty} e^{-\tau \lambda} e^{-\frac{t^2}{4\tau}} \frac{1}{\tau^{1-s}}\, d\tau\,  dE(\lambda) \varphi.$$
In the integral in $\tau$ let us set $\lambda \tau = \mu$. We obtain,
\begin{equation}\label{newu}
 u(t,x) =  \int_0^{+\infty} \theta(\lambda,t) dE(\lambda) \varphi, \quad \theta(\lambda,t)  = \frac{1}{\Gamma(s)}\int_0^{+\infty} e^{-\mu} e^{-\frac{\lambda\,t^2}{4\mu}}\frac{d\mu}{\mu^{1-s}}
 \end{equation}
  Since $ \frac{e^{-\mu}}{\mu^{1-s}} \in L^1((0, +\infty))$ it is easy to see that, for fixed $\lambda,$ the function $t\mapsto  \theta(\lambda,t) $ is continuous. Let $t_0\in [0, +\infty)$ and $(t_j)\subset [0, +\infty)$ a sequence converging to $t_0$. Then,
$$\Vert u(t_j , \cdot) - u(t_0, \cdot) \Vert^2_{L^2(\xR^d)} =  \int_0^{+\infty}\vert  \theta(\lambda,t_j) -\theta(\lambda,t_0)\vert^2 d(E(\lambda)\varphi, \varphi).$$
Since $\theta(\lambda,t) \leq 1$ and $\varphi \in L^2(\xR^d)$ the Lebesgue theorem shows that the right hand side tends to zero when $j$ goes to $ +\infty$.

 Now, using the fact (proved by induction on $k\in \xN$) that,
   $t^k\partial_t^ke^{-\alpha t^2}  = \sum_{j=1}^k c_{jk}(\alpha t^2)^je^{-\alpha t^2},$ $\alpha>0$,    $c_{jk} \in \xR,$  we see that,
\begin{equation}\label{newdku}
t^k\partial^k_t u(t,x) =    \int_0^{+\infty} \theta_k(\lambda,t) dE(\lambda) \varphi, \quad \theta_k(\lambda,t)  = \frac{1}{\Gamma(s)} \sum_{j=1}^k d_{jk}\int_0^{+\infty}e^{-\mu} \big(\frac{\lambda\,t^2}{4\mu}\big)^j e^{-\frac{\lambda\,t^2}{4\mu}}\frac{d\mu}{\mu^{1-s}}.
\end{equation}
Since $ \vert \theta_k(\lambda,t) \vert \leq M_k, M_k \in \xR^+,$ the same argument as before shows that $t^k \partial_t^ku \in C^0([0, +\infty[, L^2(\xR^d))$. Thus $\partial_t^ku \in C^0(]0, +\infty[, L^2(\xR^d))$.

Suppose   we have shown that   $u$  satisfies equation \eqref{ext-pb}, then  $u$  is a $C^\infty$ function on $]0, +\infty[ \times \xR^d$.  Indeed the operator appearing in \eqref{ext-pb} is an operator "sum of squares" with $C^\infty$ coefficients, which by hypothesis satisfies H\" ormander's condition \eqref{lie}. It is then hypoelliptic. So we are left with the proof of \eqref{ext-pb}.

According to the computation made before we see that,
\begin{align*}
 \partial_t^2 u + \frac{1-2s}{t}\partial_t u &= \frac{1}{\Gamma(s)} \int_0^{+\infty}\int_0^{+\infty}e^{-\mu}\lambda \big( -\frac{1-s}{2\mu} + \frac{\lambda t^2}{4 \mu^2}\big)e^{-\frac{\lambda\,t^2}{4\mu}} \frac{1}{ \mu^{1-s}}\, d\mu\, dE(\lambda)\varphi,\\
 &=\frac{1}{\Gamma(s)}  \int_0^{+\infty}\int_0^{+\infty}\lambda  \, e^{-\mu}\frac{\partial}{\partial \mu}\big[e^{-\frac{\lambda\,t^2}{4\mu}}\frac{1}{ \mu^{1-s}}\big]\, d\mu\, dE(\lambda)\varphi,\\
 &=  \frac{1}{\Gamma(s)}  \int_0^{+\infty}\int_0^{+\infty}\lambda \, e^{-\frac{\lambda\,t^2}{4\mu}}\frac{1}{ \mu^{1-s}}\, d\mu\, dE(\lambda)\varphi.
 \end{align*}
Setting $\mu = \lambda \tau$ we deduce that,
\begin{align*}
 \partial_t^2 u + \frac{1-2s}{t}\partial_t u &= \frac{1}{\Gamma(s)} \int_0^{+\infty}\int_0^{+\infty} \lambda^{s+1} e^{-\lambda \tau} e^{-\frac{t^2}{4 \tau}}\frac{1}{\tau^{1-s}}\, d\tau \, dE(\lambda)\varphi,\\
 &= \frac{1}{\Gamma(s)} \int_0^{+\infty} e^{-\tau P}(P^{s+1}\varphi) \, e^{-\frac{t^2}{4 \tau}}\frac{1}{\tau^{1-s}}\, d\tau = Pu.
\end{align*}
Let us prove \eqref{limite}. According to \eqref{newu} we have,
$$t^{1-2s} \partial_t u  = \frac{1}{\Gamma(s)} \int_0^{+\infty}\int_0^{+\infty}e^{-\mu} \frac{-\lambda t}{2 \mu} e^{-\frac{\lambda t^2}{4 \mu}} \frac{t^{1-2s}}{\mu^{1-s}}\, d\mu\, dE(\lambda)\varphi.$$
Setting $x = \frac{\mu}{\lambda t^2}$ we get,
$$t^{1-2s} \partial_t u = \frac{-1}{2\Gamma(s)}  \int_0^{+\infty}\int_0^{+\infty} e^{- \lambda x t^2} \lambda^s \frac{e^{- \frac{1}{4x}}}{x^{2-s}}\, dx \,dE(\lambda) \varphi. $$
Since $s<1$, we have $\frac{e^{- \frac{1}{4x}}}{x^{2-s}} \in L^1((0,+\infty)).$ Set $D(s) = \int_0^{+\infty} \frac{e^{- \frac{1}{4x}}}{x^{2-s}}\, dx $  and $C(s) = -\frac{D(s)}{2\Gamma(s)}.$ Now since $P^s \varphi \in L^2(\xR^d)$ we can write, $P^s \varphi = \int_0^{+\infty} \lambda^s dE(\lambda) \varphi$ so that,
$$t^{1-2s} \partial_t u -  C(s) P^s \varphi =  \frac{-1}{2\Gamma(s)} \int_0^{+\infty}\lambda^s \Big(\int_0^{+\infty} (e^{- \lambda x t^2} -1)\frac{e^{- \frac{1}{4x}}}{x^{2-s}}\, dx \Big) \, dE(\lambda) \varphi.$$
Set $f(\lambda,t) = \int_0^{+\infty} (e^{- \lambda x t^2} -1)\frac{e^{- \frac{1}{4x}}}{x^{2-s}}\, dx. $ By the dominated convergence theorem we have, for fixed $\lambda \in (0, +\infty),$ $\lim_{t\to 0^+} \vert f(\lambda,t)\vert = 0$ and moreover $\vert f(\lambda,t)\vert \leq 2D(s).$ Then,
$$\Vert t^{1-2s} \partial_t u -  C(s) P^s \varphi \Vert^2_{L^2(\xR^d)} = \frac{1}{4\Gamma(s)^2}  \int_0^{+\infty} \lambda^{2s} \vert f(\lambda,t)\vert^2 d\big(E(\lambda)\varphi, \varphi),$$
and,  by the dominated convergence theorem, the right hand side tends to zero when $t$ goes to $0^+.$

\end{document}